\documentclass[12pt]{article}
\usepackage{amsmath,amssymb,amsthm,graphicx}
\usepackage{enumitem}
\usepackage{hyperref}
\usepackage{fullpage}
\usepackage{float}
\usepackage{setspace}

\newtheorem{theorem}{Theorem}
\newtheorem{lemma}[theorem]{Lemma}

\newtheorem{corollary}[theorem]{Corollary}

\newcommand{\z}{+o(1)}

\begin{document}

\title{Cops and Robbers on graphs of bounded diameter}

\author{
Seyyed Aliasghar Hosseini\hspace{5mm} Fiachra Knox \hspace{5mm} Bojan Mohar\thanks{Supported in part by an NSERC Discovery Grant R611450 (Canada),
   by the Canada Research Chair program, and by the
    Research Grant J1-8130 of ARRS (Slovenia).}~\thanks{On leave from:
    IMFM \& FMF, Department of Mathematics, University of Ljubljana, Ljubljana,
    Slovenia.}\\[5mm]
Department of Mathematics, Simon Fraser University\\
Burnaby, BC V5A 1S6, Canada
}

\date{\today}

\maketitle

\begin{abstract}
The game of Cops and Robbers is a well known game played on graphs. In this paper we consider the class of graphs of bounded diameter. We improve the strategy of cops and previously used probabilistic method which results in an improved upper bound for the cop number of graphs of bounded diameter. In particular, for graphs of diameter four, we  improve the upper bound from $n^{\frac{2}{3}+o(1)}$ to $n^{\frac{3}{5}+o(1)}$ and for diameter three from $n^{\frac{2}{3}+o(1)}$ to $n^{\frac{4}{7}+o(1)}$. 
\end{abstract}

\section{Introduction}
The game of Cops and Robbers was  introduced by Nowakowski and Winkler \cite{nowakowski} and Quilliot \cite{quilliot}. There are several versions of the game and the original one has the following rules.
There are two players, one of them controls the cops and the other one controls the robber. The game is played on the vertices of a graph $G$. At the beginning of the game, each cop will choose a vertex as the \emph{initial position}\index{initial position} and then the robber will choose a position. In each further step of the game, first each cop \emph{moves}\index{move}, where to move means either staying at the same position or changing the position to a neighbor of its current position. After that, the robber moves. Several cops can occupy the same position at any time. 
This is a perfect information game, meaning that each player has full knowledge about the playground (the graph $G$) and positions of all players. 

The cops win the game if, for every strategy of the robber of selecting the initial position and moves, they \emph{catch}\index{catch} the robber, meaning that after their move some cop has the same position as the robber. Otherwise, the robber wins the game.
The smallest number $k$ of cops for which the
cops win the game is called the \emph{cop number}\index{cop number} of $G$ and is denoted by $c(G)$. 

One of the main open questions in this area is Meyniel's conjecture, which asks whether the cop number of any connected graph on $n$ vertices can be bounded by $O(\sqrt{n})$. For more information see \cite{Bonato} and references within.

Lu and Peng \cite{lu-peng} (and independently Scott and Sudakov \cite{scott-sudakov}) proved the following theorem which gives the best known upper bound of the cop number of general graphs.

\begin{theorem} \label{thm:lu-peng}
The cop number of any connected $n$-vertex graph is at most $n2^{-(1+o(1))\sqrt{\log{n}}}$.
\end{theorem}

The logarithm in Theorem \ref{thm:lu-peng} as well as all other logarithms in this paper are taken base 2.
The following is a specific case of their results.

\begin{corollary}
The cop number of any connected $n$-vertex graph of diameter $d$ is at most $n^t$, where $t= {1-\frac{1}{\lceil \log{d}\rceil +1}+o(1)}$.
\end{corollary}

Our Theorem \ref{main} gives an improvement of this bound for $d\geq 4$, while for graphs of diameter 3, we have an even better bound -- see Theorem \ref{thm:3}. Our third main result is Theorem \ref{high_girth}, which gives an improved bound of Theorem \ref{main} when the girth of the graph is large. At the end of the paper we also add a new bound on the cop number of digraphs of small diameter, see Theorem \ref{thm:dir}.

In the proofs we use a novel strategy based on ``real and imaginary cops''. We first consider a strategy where a large set of ``imaginary cops'' is used. We show that these cops are able to surround the robber in such a way that many of them will catch the robber at the same time. When there are only polynomially many scenarios for capturing the robber, we show that there exists a subset of the imaginary cops (and these are the ``real cops'') such that in each scenario, at least one of the real cops will participate in capturing the robber.  We also show that it is possible to repeat the basic probabilistic strategy of \cite{lu-peng} and \cite{scott-sudakov} over and over again such that after each iteration we shrink the size of the neighborhood, where the robber can move without being captured.

\section{Real and imaginary cops}

Lu and Peng in \cite{lu-peng} and Scott and Sudakov in \cite{scott-sudakov} used random positioning\index{random position} of cops to analyze the game. We will restate their starting tool in a more general language.

Let $\mathcal{C}=\mathcal{C}(V,p)$ be a random subset of a set $V$ with $|V|=n$, where each $v\in V$ is in $\mathcal{C}$ with probability $p$, independently from other elements in $\mathcal{C}$. Since $|\mathcal{C}|$  is binomially distributed\index{binomially distribution} with expectation\index{expectation} $\mu=n\cdot p$, by the standard Chernoff-type estimate\index{Chernoff-type estimate}, the probability that $\mathcal{C}$ has more than $2\mu =2np$ vertices is at most $e^{-\mu/3}$.

For every subset $A$ of vertices of $G$ and each integer $i$, let $B(A,i)$ be the ball of radius $i$ around $A$, that is the set of all vertices of $G$ that can be reached from some vertex in $A$ by a path of length at most $i$. For simplicity, when $A$ is a single vertex $v$, we write $B(v,i)$. We need the following lemma.

\begin{lemma}\label{lem:tool}
Let $G$ be a connected graph of order $n$ and let $\mathcal{C}=\mathcal{C}(V(G),p)$. For every $n\geq 333$, the following statement holds with probability at least $0.9$: For every $A\subset V(G)$ and every $i$ such that $|B(A,i)|\geq |A|\frac{\log^2{n}}{p}$, we have 
\begin{equation}
    |B(A,i)\cap \mathcal{C}|\geq |A|.
    \label{ball-ineq}
\end{equation}
\end{lemma}

\begin{proof}
Let $a=|A|$. Note that for fixed $A$ and $i$, the number of
vertices in $\mathcal{C}\cap B(A,i)$ is binomially distributed
with expectation at least $a\frac{\log^2{n}}{p}\cdot p=a \log^2{n}$, and by the standard
Chernoff-type estimate we have that the probability that
$|B(A,i)\cap \mathcal{C}|<a$  is at most $e^{-a\log^2{n}/3}$. The number of sets of size
$a$ is ${{n}\choose{a}}$ and the number of different choices of $i$
is at most the diameter of the graph, $d(G)$. Since inequality \ref{ball-ineq} holds for $a=0$, the
statement is true with probability at least
\begin{center}
$1-d(G)\sum_{a=1}^n {{n}\choose{a}} e^{-a\log^2{n}/3}.$    
\end{center}

Since $${{n}\choose{a}} \leq \frac{n^a}{a!}\leq e^{a\log n}$$
we have
\begin{center}
$1-d(G)\sum_a {{n}\choose{a}} e^{-a\log^2{n}/3}\geq 1- d(G) \sum_a e^{a\log{n} - a\log^2{n}/3}\geq 1- d(G) \sum_a e^{-a\log^2{n}/6}.$
\end{center}
The last inequality holds when $\log{n} - \log^2{n}/3 \leq -\log^2{n}/6$ which is true for $n\geq e^3$. So the statement is true with probability at least
\begin{center}
$ 1- d(G)\sum_a (e^{-\log^2{n}/6})^a\geq 1- d(G) \frac{e^{-\log^2{n}/6}}{1-e^{-\log^2{n}/6}}=1- d(G) \frac{1}{e^{\log^2{n}/6}-1}.$
\end{center}

Note that $d(G)<n$ and $e^{\log^2{n}/6}>2^{\log^2{n}/6}=n^{\log{n}/6}$. Thus, in order to guarantee the truth of the statement with probability at least 0.9, it satisfies to have
\begin{center}
$1- d(G) \frac{1}{e^{\log^2{n}/6}-1}>1- \frac{n}{n^{\log{n}/6}-1}>0.9,$
\end{center}
which is true for $n\geq 333$.
\end{proof}

Although the above proof is using  properties of random sets, it is important to observe that there exist an actual set $\mathcal{C}$ with the desired property.

\begin{corollary} \label{cor:1}
For any $0<p\leq 1$, there is a set $\mathcal{C}\subseteq V(G)$ with $|\mathcal{C}|\leq 2np$ such that for every $i\in \mathbb{N}$ and for every $A\subseteq V(G)$ with $|B(A,i)|\geq |A|\frac{\log^2{n}}{p}$ we have $|B(A,i)\cap \mathcal{C}|\geq |A|$.
\end{corollary}

Now let us select a set $\mathcal{I}$ randomly from vertices of $G$  which has the property of Corollary \ref{cor:1} for probability $p_1$. Next, from vertices in $\mathcal{I}$ we select a subset $\mathcal{R}$ by taking each element of $\mathcal{I}$ with probability $p_2$.
In Lemma \ref{lem:tool} and Corollary \ref{cor:1}, $B(A,i)$ can be changed to $B(A,i)\cap \mathcal{I}$ and we obtain the following corollary.

\begin{corollary}\label{cor:2}
Let $\mathcal{I}$ be a fixed set of vertices of $G$ and let $0< p \leq 1$. There is a set $\mathcal{R}\subseteq \mathcal{I}$ with $|\mathcal{R}|\leq 2|\mathcal{I}|p$ such that for every $i\in \mathbb{N}$ and for every $A\subseteq V(G)$ with $|B(A,i)\cap \mathcal{I}|\geq |A|\frac{\log^2{n}}{p}$ we have $|B(A,i)\cap \mathcal{R}|\geq |A|$.
\end{corollary}

\section{Graphs of diameter four}
\label{diam4}

Let $G$ be a graph of diameter at most 4 on $n$ vertices.  Lu and Peng \cite{lu-peng} proved that the cop number of every such graph is at most $n^{\frac{2}{3}+o(1)}$. In this section  we will first improve this result to $n^{\frac{5}{8}+o(1)}$ and then to $n^{\frac{3}{5}+o(1)}$.

Let $\mathcal{C}$ be a random subset of vertices of $G$, where a vertex $v$ is in $\mathcal{C}$ with probability $p=n^{-\frac{3}{8}}$. As discussed above, $\mathcal{C}$ has less than $2\mu=2np = 2n^{\frac{5}{8}}$ vertices with probability at least $1-e^{np/3}$. We will put one cop on each of the vertices in $\mathcal{C}$.

Although we speak of a random set $\mathcal{C}$, what we mean is a concrete set $\mathcal{C}\subseteq V(G)$ that satisfies the condition of Corollary \ref{cor:1}.

Let $r$ be the position of the robber. If the size $|B(r,1)|$ of the neighborhood of $r$ is greater than $n^{\frac{3}{8}}\cdot \log^2{n}$, then by the property of Corollary \ref{cor:1} there is a cop in the robber's neighborhood who will capture the robber at the very beginning. So we may assume that $|B(r,1)|<n^{\frac{3}{8}}\cdot \log^2{n}$. 

Consider the bipartite graph $H$ with partition classes $B(r,1)$ and $B(r,3)\cap \mathcal{C}$ defined as follows. The edge $uv$ exists in $H$ if and only if there is a path of length at most 2 between (the corresponding vertices) $u$ and $v$ in $G$. Note that $B(r,3)\cap \mathcal{C}$ might be empty and also we may assume that $B(r,1)$ and $B(r,3)\cap \mathcal{C}$ are disjoint sets (a cop in $B(r,1)$ can capture the robber immediately). If we can move some cops from $B(r,3)$ in at most 2 moves to occupy all vertices of $B(r,1)$, then there is a matching in $H$ that covers all vertices of $B(r,1)$. This means that in one move, the cops can guard $B(r,1)$ and the robber cannot move and therefore will be captured. 

So we may assume that this matching does not exist. Therefore, by Hall's Theorem, there is a set $S_1 \subseteq B(r,1)$ such that $|S_1| > |N_H(S_1)| = |B(S_1,2) \cap \mathcal{C}|$. Furthermore, by Hall's Theorem there exist a set $S_1$ such that there is a matching in $H$ covering $B(r,1)\setminus S_1$. Therefore the cops in $\mathcal{C}\cap B(S_1,2)$ can move so that all vertices in $B(r,1) \setminus S_1$ will be guarded by cops after their move, and the robber cannot use them or will get caught.
As a consequence $|B(S_1,2) \cap \mathcal{C}| < |B(r,1)|$. If $|B(S_1,2)| \geq n^\frac{3}{8} |B(r,1)| \cdot \log^2{n}$, then by Lemma \ref{lem:tool}, $|B(S_1,2) \cap \mathcal{C}|\geq |B(r,1)|$. This would be a contradiction. So we may assume that $|B(S_1,2)| < n^\frac{3}{8} |B(r,1)| \cdot \log^2{n} \leq n^{\frac{3}{4}}\cdot \log^4{n}$.

\begin{figure}[H]
    \centering
    \includegraphics[width=0.8\textwidth]{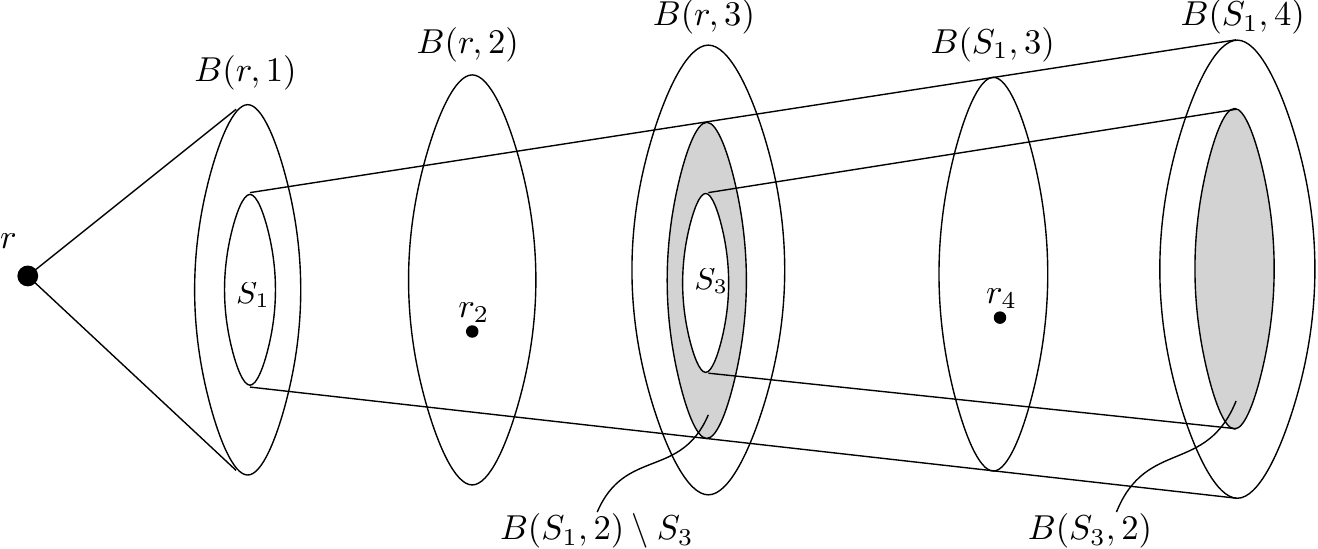}
    \caption{The neighborhood of $r$ as the robber moves. The sets shown are not disjoint. Note that $B(r,2)\subseteq B(r,3)$ and $B(r,2)\subseteq B(S_1,3)$, etc. The robber will be in the $i^{th}$ set $B(r,i)$ after $i$ moves and if goes to the set $S_1$ in the first step, will be in the set $B(S_1,i-1)\subseteq B(r,i)$.}
    \label{fig:main}
\end{figure}

Let us mention that this basically proves that $c(G)\leq n^{\frac{3}{4}+o(1)}$. Note that since the diameter is 4, any cop can move to any vertex in $B(S_1,2)$ in 4 moves and also can guard it (get to one of its neighbors) in 3 moves. Therefore, since $|B(S_1,2)|<n^{\frac{3}{4}+o(1)}$, we can move our (additional set of $n^{\frac{3}{4}+o(1)}$) cops and in 3 moves guard all vertices in $B(S_1,2)$, before the robber is able to reach any vertex in $B(S_1,2)\setminus B(S_1,1)$, through which he could have escaped out of $B(S_1,2)$.
Also if we would consider $p=n^{-\frac{1}{3}}$, then we would be able to argue that  $c(G)<n^{\frac{2}{3}+o(1)}$.
This strategy has been used in the work of Lu and Peng \cite{lu-peng} and also by Scott and Sudakov \cite{scott-sudakov}. However, there is some improvement possible that was not discovered in \cite{lu-peng, scott-sudakov}.

Our plan is to send some cops from $B(S_1,4)$ to occupy vertices in $B(S_1,2)$. 
Let $\mathcal{I}$ be a random subset of vertices of $G$, where a vertex $v$ is in $\mathcal{I}$ with probability $n^{-1/8} \cdot \log^8{n}$. The probability that  $|\mathcal{I}|$ is smaller than $n^{7/8} \cdot \log^6{n}$ is at most $e^{-n^{7/8} \cdot \log^8{n}\,/\,3}$. We will call vertices in $\mathcal{I}$ \emph{imaginary cops}\index{imaginary cop}. Also let $\mathcal{R}$ be a random subset of vertices of $\mathcal{I}$, where a vertex $v$ is in $\mathcal{R}$ (independently uniformly at random) with probability $p'=n^{-\frac{1}{4}}$. We would like to put one (real) cop on each vertex of $\mathcal{R}$. Note that $|\mathcal{R}|>2 |\mathcal{I}|p'=2n^{\frac{5}{8}} \cdot \log^6{n}$ with probability at most $e^{-n^{5/8} \cdot \log^6{n}\,/\,3}$.
Also note that the (real) cops that we are using in this step are different from those used in the first step.


Now we consider the bipartite graph $H'$ with parts $B(S_1,2)$ and $B(S_1,4)\cap \mathcal{I}$ defined in the same way as $H$ was defined above (if there is a vertex in both parts remove it from $B(S_1,2)$ to make the two sets disjoint).
If there is a matching in $H'$ between vertices in $B(S_1,2)$ and the imaginary cops in $B(S_1,4)$ then we can occupy $B(S_1,2)$ with imaginary cops in two moves. Otherwise, using the previous argument and Hall's Theorem\index{Hall's Theorem}, there is a set $S_3 \subseteq B(S_1,2)$ such that $|S_3| > |B(S_3,2) \cap \mathcal{I}|$. Thus, $|B(S_3,2) \cap \mathcal{I}| < |S_3| \leq |B(S_1,2)|$ and there is a matching between $B(S_3,2) \cap \mathcal{I}$ and $B(S_1,2)\setminus S_3$. Now by Lemma \ref{lem:tool}, if $|B(S_3,2)| \geq n^\frac{1}{8} |B(S_1,2)|\cdot \log^2{n}$, then  $|B(S_3,2) \cap \mathcal{I}| \geq |B(S_1,2)|$, which is a contradiction. Therefore we have  $|B(S_3,2)| < n^\frac{1}{8} |B(S_1,2)|\cdot \log^2{n} < n^\frac{7}{8}\cdot \log^6{n}$. 

Note that since we have at least $n^\frac{7}{8}\cdot \log^6{n}$ imaginary cops and the set $B(S_3,2)$ is at distance at least 3 from $r$, we will have enough time to get our (imaginary) cops to occupy all vertices in $B(S_3,2)$ (since the diameter is four and the cops start the move, this is possible). Therefore after two (case 1) or four moves (case 2), the robber's entire neighborhood will be occupied by imaginary cops.

Let us discuss these two cases separately. At the beginning, the robber is in $r$ and is forced to move to $S_1$ and then to a vertex $r_2$ in $B(S_1,1)$. Now we need to show that the robber's neighborhood in $B(S_1,2)\setminus S_3$ is small (case 1). Note that (since there is a matching between $B(S_3,2) \cap \mathcal{I}$ and $B(S_1,2)\setminus S_3$) when the robber gets to $r_2$, its neighborhood in $B(S_1,2)\setminus S_3$ is occupied by imaginary cops. 
\vspace{2mm}

\textbf{Claim 1.} If $X=(B(S_1,2)\setminus S_3)\cap B(r_2,1)$ has more than $n^\frac{1}{4}\cdot \log^2{n}$ vertices, then at least one of the imaginary cops in $X$ is real and can capture the robber.

\begin{proof}[Proof of Claim 1]
The number of (real) cops in $X$ is binomially distributed with probability $p'=n^{-\frac{1}{4}}$ and with expectation at least 
$n^\frac{1}{4}\cdot \log^2{n}\cdot n^{-\frac{1}{4}} = \log^2{n}$. By the standard Chernoff-type estimate we know that the probability of having no (real) cop in $X$ is at most $e^{-\log^2{n}/3}$. 
We have at most $n$ choices for $r_2$ and since the sets $\mathcal{C}$ and $\mathcal{I}$ have been selected before, the sets $S_1$ and $S_3$ can be uniquely determined by the position $r$ of the robber. Since we have $n$ different choices for $r$, the statement is true with probability at least 
$$1-n^2\cdot e^{-\log^2{n}/3},$$
which is greater than 0.9 for $n\geq 1085$. This confirms that the claim is true with high probability. In other words, we can select our set $\mathcal{R}$ from vertices in $\mathcal{I}$ in a way that Claim 1 holds.
\end{proof}

By the above claim, if the robber wants to move to $B(S_1,2)\setminus S_3$, then, since it is occupied by imaginary cops, $$|X| < n^\frac{1}{4}\cdot \log^2{n}.$$ 

On the other hand, the robber can decide to move to $S_3$ and then to a vertex $r_4$ in $B(S_3,1)$ where $B(r_4,1)\subseteq B(S_3,2)$ (case 2). Note that the sets $S_1$ and $S_3$ are determined by $r$ and the position of cops. In order to get to $r_4$, the robber needs to move 4 times. Since the diameter of the graph is 4, we have enough time to move our (real) cops to any position in $B(S_3,2)$ that we want. By Corollary \ref{cor:2}, there is a set $\mathcal{R}'$ in $B(S_3,2)$ for our cops such that if $|B(r_4,1)|\geq \log^2{n}\cdot |B(S_3,2)|/|\mathcal{R}'|$, then $|B(r_4,1)\cap \mathcal{R}'|\geq 1$. Therefore, if  $|B(r_4,1) > n^\frac{1}{4}\cdot \log^2{n}$, then there is a real cop there that can catch the robber.

So in both cases, we get to a point where the first neighborhood of the robber is of size at most $n^\frac{1}{4}\cdot \log^2{n}$. 
Repeating the argument of the first step for $r_2$ or $r_4$ (instead of $r$), the new set $S_1$ corresponding to $r_2$ (or $r_4$) is smaller than $n^\frac{5}{8}\cdot \log^4{n}$. Since we have enough time, we can move another set of $n^{\frac{5}{8}+o(1)}$ (real) cops to occupy $B(S_1,2)$ and therefore we can capture the robber. Therefore the cop number of a graph of diameter at most 4 is bounded above by $n^{\frac{5}{8}+o(1)}$.

\subsection{Repeating the strategy}

In the previous section we introduced an approach to catch the robber with $n^{\frac{5}{8}+o(1)}$ cops in a graph with diameter 4. In this section we want to improve this result to $n^{\frac{3}{5}+o(1)}$ and in order to do that let us repeat the same process in a more general manner. For simplicity, we will drop the poly-log factors. Assuming that $n$ is sufficiently large and since the exponent $k$ of any $\log^k{n}$ factor will not depend on $n$, this can be covered by the $o(1)$ term in the exponent of $n$.

Assume that we have $n^{1-\alpha \z}$ cops and they are randomly positioned throughout the graph. Therefore the probability that a vertex contains a cop is $p=n^{-\alpha}$. Using the same argument as above, it is easy to see that $|B(r,1)|<n^{\alpha \z}$ or with high probability we will catch the robber in the first round.

Define $S_1$ and $B(S_1,2)$ as above and similarly conclude that $|B(S_1,2)\cap \mathcal{C}|<|B(r,1)|$ and therefore $|B(S_1,2)|<n^{2\alpha \z}$.

Now consider $n^{1-\gamma_1 \z}$ imaginary cops (each selected with probability $n^{-\gamma_1}$) where each imaginary cop is a real cop with probability $n^{\gamma_1-\alpha}$. Again by using the same strategy, we conclude that $|B(S_3,2)|<n^{2\alpha+\gamma_1 \z}$ and if $2\alpha+\gamma_1\leq 1-\gamma_1$, then we can occupy the whole $B(S_3,2)$ by imaginary cops. Therefore, by selecting $\gamma_1 = \frac{1-2\alpha}{2}$, the robber will face a neighborhood which is occupied by imaginary cops and hence the density of real cops in the robber's neighborhood is $n^{\gamma_1-\alpha \z}$. Update $r$ and the definition of used sets based on the new position of the robber.

Now we can repeat the strategy again to get a better density of cops\index{density}. Assuming that we have $n^{1-\gamma_2 \z}$ (new) imaginary cops and each imaginary cop is a real cop with probability $n^{\gamma_2 - \alpha}$, we will have:
$$|B(r,1)|<n^{\alpha -\gamma_1 \z},\hspace{5mm} |B(S_1,2)|<n^{2\alpha -\gamma_1 \z}, \hspace{5mm} |B(S_3,2)|<n^{2\alpha -\gamma_1+\gamma_2 \z}$$
and if $2\alpha - \gamma_1 + \gamma_2 \leq 1-\gamma_2$, then the new set of imaginary cops can occupy $B(S_3,2)$ to get a better density. We can select $\gamma_2 = \frac{1-2\alpha}{2} + \frac{1-2\alpha}{4}$. We can continue doing this and get 
$$\gamma_i = (1-2\alpha)\sum_{j=1}^i \frac{1}{2^j},$$ therefore $\gamma_i < 1-2\alpha$. When $i$ is large enough, we have $\gamma = 1-2\alpha - o(1)$.

In this part of our strategy we need to capture the robber with real cops. We have $|B(r,1)|<n^{\alpha-\gamma \z}$, $|B(S_1,2)|<n^{2\alpha -\gamma \z}$ and we want $|B(S_1,2)|<n^{1-\alpha \z}$ to be able to occupy $B(S_1,2)$ with real cops. So we have the following two conditions to hold:
\begin{equation}\label{eq:conditions on alpha}
\gamma = 1-2\alpha-o(1) \hspace{2mm} \text{and} \hspace{2mm} 2\alpha-\gamma \leq 1-\alpha.
\end{equation}
Combining these conditions we get $\alpha = \frac{2}{5} - o(1)$ which means that it suffices to have  $n^{\frac{3}{5}+o(1)}$ cops. Therefore, $c(G)< n^{\frac{3}{5}+o(1)}$ when $G$ is a graph of diameter at most 4.

Note that we need to repeat the strategy many times (depending on how close to $n^{\frac{3}{5}}$ we want to get). Although we might need to repeat on and on, we need only seven groups of (real) cops. In particular, every time we only need that the first neighbourhood of $r$ becomes smaller and smaller. So after three steps of the game, the cops used in earlier steps can be released and in at most four steps (because each related cop can reach the intended position in $d(G)\leq 4$ steps) they can take the role of the cops in that later step.

\begin{theorem}
Let $G$ be a graph of diameter at most 4. Then $c(G)<n^{\frac{3}{5}+o(1)}$.
\end{theorem}

\section{Graphs of diameter at most 3}

Let $G$ be a graph of diameter 3. Previous section shows that $c(G)< n^{\frac{3}{5}+o(1)}$. In this section we will improve this bound.

The proof in the previous section gives a strategy involving $n^{1-\alpha}$ cops. Under that strategy we make 2 or 4 steps (depending on the robber's moves), after which the robber has fewer and fewer available neighbors where he can move without being caught. By repeating this sufficiently long time, the set of possible neighbors of the robber, which we simply treat as $B(r,1)$ (where $r$ is the current position) becomes smaller than $n^{\alpha-\gamma \z}$, and this upper bound keeps being valid after every 2 or 4 steps. We may assume this happens at even steps of the game. By using another set of cops, we achieve the same property at odd steps. It is easy to see that in every four consecutive steps we encounter a situation that two consecutive positions, $r$ and $r'$, will have $|B(r,1)|<n^{\alpha-\gamma \z}$ and $|B(r',1)|<n^{\alpha-\gamma \z}$ for every $r'\in B(r,1)$. Having the current position $r$ of the robber and knowing $B(r,1)$, we can take $n^{2\alpha -2\gamma \z}$ cops and, since the diameter of $G$ is 3, we can bring in two steps one cop to a neighbor of each vertex in $B(r',1)$, for each $r'\in B(r,1)$. This means that, whichever position $r'\in B(r,1)$ the robber moves to, its neighborhood will be completely guarded, and he will be caught.

By choosing $\gamma = 1-2\alpha-o(1)$ and $\alpha = \frac37 - o(1)$, the condition $2\alpha -2\gamma \le 1-\alpha$ is satisfied, thus we have enough cops to complete the task in the very last steps of the strategy. 
This gives the following improved bound for graphs of diameter 3.

\begin{theorem}\label{thm:3}
Let $G$ be a graph of diameter at most 3. Then $c(G)<n^{\frac{4}{7}+o(1)}$.
\end{theorem}

\section{The general case}

Now let us consider the general case where the diameter of the graph is $d$.
We would like to find an upper bound for the cop number in terms of $n$ and $d$.

\begin{theorem}\label{main}
Let $G$ be a graph of diameter $d$. Then $c(G)\leq n^t$, where $t={1-\frac{2}{2\lceil \log{d} \rceil +1}+o(1)}.$
\end{theorem}

\begin{proof}
Assume that we have $n^{1-\alpha \z}$ cops randomly positioned throughout vertices of $G$. Using a similar argument as in Section \ref{diam4} we will get:
$$|B(r,1)|<n^{\alpha \z}, \hspace{2mm} \ldots ,\hspace{2mm} |B(S_{2^k-1},2^k)|<n^{(k+1)\alpha \z} \hspace{4mm} \text{for} \hspace{2mm} k=1, \ldots ,\lceil \log{d} \rceil-1.$$
 Now assume that we have a set $I$ of $n^{1-\gamma \z}$ imaginary cops where each of them is a real cop with probability $n^{\gamma-\alpha}$. 

Let $k=\lceil \log{d} \rceil-1$. If there is a matching between vertices in $B(S_{2^k-1},2^k)$ and the imaginary cops in $B(S_{2^k-1},2\cdot 2^k)$ then we can occupy $B(S_{2^k-1},2^k)$ with imaginary cops in $2^k$ moves (here we consider the matching with respect to the bipartite graph whose edges correspond to paths of length $\leq 2^k$ in $G$). Otherwise, using the previous argument and Hall's Theorem, there is a set $S_{2^{k+1}-1} \subseteq B(S_{2^k-1},2^{k})$ such that $|S_{2^{k+1}-1}| > |B(S_{2^{k+1}-1},2^{k}) \cap I|$. 
Thus $$|B(S_{2^{k+1}-1},2^{k}) \cap I| < |B(S_{2^k-1},2^k)|,$$ and therefore by using Lemma \ref{lem:tool}, $|B(S_{2^{k+1}-1},2^{k})| < n^{\gamma \z} |B(S_{2^k-1},2^k)| < n^{(k+1)\alpha+\gamma \z}$. Note that we have $n^{1-\gamma \z}$ (imaginary) cops and we consider the set $B(S_{2^{k+1}-1},2^{k})$ after (at least) $d$ steps of the robber. If $1-\gamma>(k+1)\alpha+\gamma$, then we can get our imaginary cops to occupy all vertices in $B(S_{2^{k+1}-1},2^{k})$ at the considered time. Therefore after $2^k$ or $2^{k+1}$ moves, the robber's entire neighborhood will be occupied by imaginary cops and the density of real cops will improve from $n^{-\alpha}$ to $n^{\gamma-\alpha}$. Thus, for the position of the robber at that time, say $r$, with high probability, we have:
$$|B(r,1)|<n^{\alpha-\gamma \z}, \hspace{2mm} |B(S_1,2)|<n^{2\alpha-\gamma \z}, \hspace{2mm} \ldots ,\hspace{2mm} |B(S_{2^k-1},2^k)|<n^{(k+1)\alpha-\gamma \z}.$$

Repeating the argument with a new set of $n^{1-\beta \z}$ imaginary cops can improve the density of real cops in the first neighborhood of the robber and we can get
$$|B(S_{2^{k+1}-1},2^{k})| < n^{(k+1)\alpha-\gamma+\beta \z}.$$
And if $(k+1)\alpha-\gamma+\beta<1-\beta$, then the new imaginary cops can improve the density. We can continue this until $\gamma=\beta$ which means that $\gamma=\beta<1-(k+1)\alpha$.

Now after having the desired density of real cops in the first neighborhood of the robber (i.e. $\gamma = \beta$), we would like to capture the robber. If $(k+1)\alpha-\gamma<1-\alpha$, then we can use another set of real cops to guard the set $B(S_{2^k-1},2^k)$ in at most $d-1\leq 2^{k+1}-1$ steps. Note that we need to guard $B(S_{2^k-1},2^k)$ when the robber enters it after $2^{k+1}-1$ moves ($2^k+2^{k-1}+\cdots +2+1$ moves) which gives the cops the required time to guard these vertices and capture the robber.

Now combining the conditions $(k+1)\alpha-\gamma<1-\alpha \hspace{3mm} \text{and} \hspace{3mm} \gamma<1-(k+1)\alpha$, we see that
$\alpha<\frac{2}{2k+3}$ works. Therefore, $c(G)\leq O(n^{1-\frac{2}{2k+3}+o(1)})$, where $k=\lceil \log{d} \rceil-1$ and $d$ is the diameter of $G$.
\end{proof}

\section{Graphs of large girth}

In this section we will improve the strategy of cops and decrease the derived bound on the cop number for graphs of large girth. Define $B'(A,i)=B(A,i)\setminus B(A,i-1)$. 

Let $G$ be a graph of girth $g$ and let $\rho=\lfloor \frac{g+1}{4} \rfloor$. The following lemma is our main tool.

\begin{lemma}\label{lem:guard}
For every vertex $u$, if the robber is not in $B(u,2\rho-2)$, then two cops can prevent the robber from entering $B(u,\rho)$.
\end{lemma}

\begin{proof}
For $\rho=1$, the robber is not in $B(u,0)$ and a cop, by staying at $u$ can prevent the robber from entering $B(u,1)$. So we may assume that $\rho\geq 2$ and therefore $g\geq 7$. 
Now let us assume that cops $C_1$ and $C_2$ are at $u$ and the robber has entered $B'(u,2\rho-1)$ and it is the cops' turn. Because of the girth condition, there is a unique vertex $x$ in $B'(u,\rho)$ that the robber can enter in $\rho-1$ moves. Therefore $C_1$ will move one step towards that vertex to be able to guard it in $\rho -1$ moves. From now on, $C_1$ will copy the movements of the robber: if the robber moves closer to $x$, then $C_1$ will moves closer to $x$ as well and if the robber moves away from $x$, then $C_1$ will move back towards $u$.
Note that when the robber is in $B'(u,2\rho-1)$, the entrance vertex in $B'(u,\rho)$ can be changed in one move. In this case, the other cop, $C_2$, will move one step towards the robber and $C_1$ will get back to $u$. So if the robber is in $B'(u,2\rho-k)$, then there is one cop in $B'(u,k)$ and one in $u$.
Therefore, by using this strategy, the two cops can prevent the robber from entering $B'(u,\rho)$.
\end{proof}

Now we have the tool to improve the result from the previous section.

\begin{theorem}\label{high_girth}
Let $G$ be a graph of diameter $d$ and girth $g$ and let $\rho=\lfloor \frac{g+1}{4} \rfloor$. Then $c(G)\leq n^t$, where $t={1-\frac{2}{2\lceil\log{(d/\rho)}\rceil+1}+o(1)}.$
\end{theorem}
\begin{proof}
We will use the same strategy as in the previous section with the exception that instead of having each cop guarding a vertex and its neighbourhood, we replace each of them with two cops at each such vertex $u$ and will make them guard $B(u,\rho)$, as shown in Lemma \ref{lem:guard}.

In the first step of our previous strategy, we used the probability of having a cop in the first neighborhood of the position of the robber to bound the size of the first neighborhood of the position of the robber. To get a better result, assume that we have  (two sets of) $n^{1-\alpha \z}$ cops and place two cops (instead of one) randomly on each vertex with probability $p$. 

Let us first assume that two cops are at $u$ but the robber is already in $B(u,2\rho-2)$. To start our strategy in Lemma \ref{lem:guard}, we need to push the robber out of $B(u,2\rho-2)$. When the robber is in $B(u,2\rho-1)$, there is a unique shortest path from $u$ to the robber's position and there is no cycle in $B(u,2\rho-2)$. Sending an extra cop to follow the robber will force the robber to move either towards $u$ and eventually get captured or to get out of $B(u,2\rho-2)$ (and enter $B'(u,2\rho-1)$).

Let $r$ be the position of the robber. After some steps, if there is a vertex in $B(r,\rho)$ that was selected to contain (two) cops, then it means that $r\in B(u,\rho)$ where $u$ contains two cops. By Lemma \ref{lem:guard}, the robber should have been captured by now. So by Lemma \ref{lem:tool} (and ignoring the $\log{n}$ term) we may assume that $|B(r,\rho)|<n^{\alpha \z}$.

In the next step (in the previous approach) we defined the set $S_1$ to be (roughly) the set of vertices that cannot be guarded by the cops in $B(r,3)$. We can redefine $S_1$ to be the set of vertices in $B(r,\rho)$ that cannot be guarded by the cops in $B(r,2\rho)$. Note that although the cops are moving first, we cannot bring cops from $B(r,3\rho)$ to cover $S_1$ (since $\rho$ can be more than 1). As usual, we can see that not only $|S_1|>|B(S_1,\rho)\cap \mathcal{C}|$, but also with high probability we have $|B(S_1,\rho)|< n^{2\alpha \z}$. Note that in the next step we can define $S_3$ in $B(S_1,\rho)$ and calculate the upper bound for $|B(S_3,2\rho)|$. The radius of the ball around $S_i$ will grow exponentially and we have
$$|B(S_{2^k+1},2^k\rho)|<n^{(k+1)\alpha \z} \quad \text{for} \quad k=0,\ldots, \lceil \log \frac{d}{\rho} \rceil.$$

Now we can follow the previous strategy to get $\alpha<\frac{2}{2k+3}$ and therefore (by replacing $k$ with $\lceil \log \frac{d}{\rho} \rceil-1$) the cop number is at most
$n^t$, where $t={1-\frac{2}{2\lceil\log{(d/\rho)}\rceil+1}+o(1)}$.
\end{proof}

\section{Digraphs of bounded diameter}

In this section we will consider digraphs of diameter two and bipartite digraphs of diameter three, which are the digraphs such that between any two vertices $u$ and $v$ there is a directed path of length at most two (for the first case) or three (for the second case). Note that such digraphs will automatically be strongly connected.

We will basically generalize the method that was introduced in \cite{diameter2} to digraphs.

\begin{lemma} \label{lem:diam2}
Let $k >0$ be an integer, $D$ be a digraph of diameter 2 or a bipartite digraph of diameter 3, and let $H$ be a sub-digraph of $D$, such that the maximum out-degree of $H$ is at most $k$. Suppose the robber is restricted to move on the edges of $H$, while the cops can move on $D$ as usual. Then $k+1$ cops can catch the robber.
\end{lemma}

\begin{proof}
 Let $r$ be the position of the robber, a vertex of out-degree $l\leq k$ with out-neighbors $v_1, \ldots, v_l$ and let $c_1, \ldots, c_{k+1}$ be our set of cops. 
 
 Let us first assume that $D$ is a digraph of diameter 2. Assign $c_i$ to cover $v_i$ (for $i=1,\ldots, l$). Since the diameter of $D$ is 2 (the diameter of $H$ can be different), each $c_i$ can get to $v_i$ in at most two moves and therefore in one move can get to an in-neighbor of it. Thus, in one move, the cops can position themselves in a way that the robber cannot use any of its out-neighbors. So the robber cannot move. Now send another cop (we have at least one cop more than the number of out-neighbors) to capture the robber.

Now let $D$ be a bipartite digraph of diameter 3 and let $V(D)=L\cup R$ be the bipartition of vertices of $D$. Move $k$ cops to $R$ and let the additional cop follow the robber and force the robber to move. Consider the position  $r$ of the robber when $r\in L$. Now the out-neighbors of $r$, $v_1, \ldots, v_l$ are in $R$. Assign $c_i$ to control $v_i$. Since the diameter of $D$ is 3 and $D$ is bipartite, there is a directed path of length at most 2 between the position of $c_i$ and $v_i$. So each $c_i$ by moving  once towards $v_i$ can guard it. Therefore the robber cannot use any of $v_i$'s (without being caught) and the cop who is following the robber will catch the robber. 
\end{proof}

For a pair $(D,H)$ where $H\subseteq D$, we define $c(D,H)$ to be the minimum number of cops that can capture the robber, when the robber is forced to move on  $H$ while the cops can move on $D$.

\begin{theorem}\label{thm:dir}
Let $D$ be a digraph of diameter 2, or a bipartite digraph of diameter 3, of order $n$. Then
$$c(D)\leq\sqrt{2n}.$$
\end{theorem}

\begin{proof}
By the definition we have that $c(D,D)=c(D)$. We will prove that $c(D,H)\leq \sqrt{2m}$, $m=|V(H)|$ for all $H\subseteq D$.

The proof will go by induction on $m$, the size of $H\subseteq D$. It is clear that $c(D,H)=1$ when $|V(H)|=1$ or $2$.
Now let $m\geq 3$ and assume that there is no vertex of out-degree greater than or equal to $\lfloor \sqrt{2m} \rfloor$. Then by Lemma \ref{lem:diam2} we are done and $c(D,H)\leq \lfloor \sqrt{2m} \rfloor$. Now assume that there is a vertex $v$ of out degree at least $\lfloor \sqrt{2m} \rfloor$. Put a stationary cop on $v$ to protect $v$ and its out-neighborhood. From now on the robber cannot use these vertices or will be captured by the stationary cop. Therefore, we can remove $v$ and its out-neighbors from $H$ to make $H'$. Note that $|V(H')|\leq m-\lfloor \sqrt{2m} \rfloor-1$. By the induction hypothesis, we have $c(D,H)\leq 1+c(D,H') \leq 1+\sqrt{2(m-\lfloor \sqrt{2m} \rfloor-1)} \leq \lfloor \sqrt{2m}\rfloor$.

Since this inequality holds for all subgraphs $H$ of $D$, then, $c(D,D)=c(D)\leq \sqrt{2n}$.
\end{proof}

Let us mention that it is not clear whether our methods used for undirected graphs can be applied to digraphs or not.

\bibliographystyle{plain}
\bibliography{references}

\begin{thebibliography}{1}

\bibitem{Bonato}
Anthony Bonato and Richard~J. Nowakowski.
\newblock {\em The game of cops and robbers on graphs}, volume~61 of {\em
  Student Mathematical Library}.
\newblock American Mathematical Society, Providence, RI, 2011.

\bibitem{lu-peng}
Linyuan Lu and Xing Peng.
\newblock On $\text{Meyniel's}$ conjecture of the cop number.
\newblock {\em Journal of Graph Theory}, 71(2):192--205, 2012.

\bibitem{nowakowski}
Richard Nowakowski and Peter Winkler.
\newblock Vertex-to-vertex pursuit in a graph.
\newblock {\em Discrete Mathematics}, 43(2):235--239, 1983.

\bibitem{quilliot}
A~Quilliot.
\newblock Jeux et pointes fixes sur les graphes.
\newblock {\em Th\`ese de 3\`eme cycle, Universit\'e de Paris}, VI:131--145,
  1978.

\bibitem{scott-sudakov}
Alex Scott and Benny Sudakov.
\newblock A bound for the cops and robbers problem.
\newblock {\em SIAM Journal on Discrete Mathematics}, 25(3):1438--1442, 2011.

\bibitem{diameter2}
Zsolt~Adam Wagner.
\newblock Cops and robbers on diameter two graphs.
\newblock {\em Discrete Mathematics}, 338(3):107--109, 2015.

\end{thebibliography}
\end{document}